\documentclass[12pt]{article}
\usepackage{amssymb}
\usepackage{amsthm}

\font\tenmsb=msbm10 scaled \magstep1
\font\sevenmsb=msbm7 scaled \magstep1
\font\fivemsb=msbm5 scaled \magstep1
\newfam\msbfam
 \textfont\msbfam=\tenmsb
 \scriptfont\msbfam=\sevenmsb
 \scriptscriptfont\msbfam=\fivemsb
\def\Bbb#1{{\fam\msbfam #1}}
\def\R{{\Bbb R}}
\def\Z{{\Bbb Z}}
\def\cy#1{\Z/{#1}\Z}

\def\smfrac#1#2{\frac{\scriptstyle#1}{\scriptstyle#2}}
\newcommand{\rta}[1]{\stackrel{#1}{\rightarrow}}
\newtheorem{thm}{Theorem}[section]
\newtheorem{lem}[thm]{Lemma}

\newtheorem{prop}[thm]{Proposition}

\newtheorem{cor}[thm]{Corollary}

\newtheorem{defn}[thm]{Definition}

\newtheorem{rmk}[thm]{Remark}
\newtheorem{rmks}[thm]{Remarks}

\newtheorem{hyp}[thm]{Hypotheses}

\title{Birationality of \'etale maps via surgery}
\author
{Scott Nollet\\
Department of Mathematics\\
Texas Christian University \\
Fort Worth, TX 76129
\and
Laurence R.~Taylor\\
Department of Mathematics\\
University of Notre Dame\\
Notre Dame, IN 46556
\and
Frederico Xavier
\thanks{Work partially supported by NSF
grant DMS02-03637.}\\
Department of Mathematics\\
University of Notre Dame\\
Notre Dame, IN 46556}
\date{}

\begin{document}

\maketitle

\begin{abstract}
We use a counting argument and surgery theory to show that if $D$ is a 
sufficiently general algebraic hypersurface in $\Bbb C^n$, 
then any local diffeomorphism $F:X \to \Bbb C^n$ of simply connected 
manifolds which is a $d$-sheeted cover away from $D$ has degree 
$d=1$ or $d=\infty$ (however all degrees $d > 1$ are possible if $F$ 
fails to be a local diffeomorphism at even a single point). 
In particular, any \'etale morphism $F:X \to \Bbb C^n$ of algebraic varieties which covers away from such a hypersurface $D$ must be birational.
\end{abstract}

\section{Introduction}

Polynomial maps $F: \Bbb C^n \to \Bbb C^n$ with nonvanishing jacobian 
have received much attention over the last 25 years, since the excellent expository article of Bass, Connell and Wright \cite{BCW} pointed out that 
the Jacobian Conjecture is still unsolved. The map $F$ has a finite degree 
$d>0$ and one formulation of the conjecture asks whether necessarily $d=1$. 
It is well-known that there is a Zariski open set $U \subset \Bbb C^n$ for 
which the restriction 
$F^{-1} (U) \to U$ is a $d$-sheet covering map \cite[Prop. 3.17]{mumford}. 
More recently, Jelonek has shown that if $D \subset \Bbb C^n$ is the closed 
set over which $F$ is not a $d$-cover, then either $D = \emptyset$ or 
$D$ is a uniruled hypersurface \cite{J}. He also gives a sharp degree bound for 
the equation of $D$ and explains how to compute its equation with a Gr\"obner basis. Of course if $D = \emptyset$, then $F$ is a covering map 
and simple connectivity of $\Bbb C^n$ implies that $d=1$ as conjectured, 
so one may study the problem in terms of the hypersurface $D$.

Kulikov \cite{kulikov2} considers the more general question of whether an 
\'etale morphism $F:X \to \Bbb C^n$ of simply connected varieties which is 
surjective modulo codimension two must be birational, i.e. have degree $d=1$. 
In terms of the divisor $D$ above, he observes that it is equivalent to ask whether any subgroup of $\pi_1 (\Bbb C^n - D)$ generated by geometric generators (loops about $D$ with winding number one) must have index one.
The answer is yes for $n=1$, but Kulikov constructs a counterexample 
for $n=2$ of degree $d=3$ by taking a quartic curve $D \subset \Bbb C^2$ with three cusps and producing a subgroup $G \subset \pi_1(\Bbb C^2 - D)$ 
of index three defined by a geometric generator \cite[\S 3]{kulikov2}. 
We will show by contrast that for general hypersurfaces 
$D \subset \Bbb C^n$, $F$ must indeed be birational. 

A hypersurface $D \subset \Bbb C^n$ is given by an equation $f(z_1,z_2,\dots,z_n)=0$. A theorem of Verdier says that the 
corresponding map $f: \Bbb C^n \to \Bbb C$ is a locally trivial fibration 
away from a finite subset of $\Bbb C$ \cite{verdier}. 
The smallest such set, the bifurcation locus $B_f \subset \Bbb C$, 
contains the image of the critical values of $f$, but may also contain 
the images of critical points at infinity: since $f$ is not proper, one cannot 
apply the Ehresmann fibration theorem. Generalizing work of H\`a Huy Vui and 
L\^e D$\sim$ ung Tr\'ang \cite{HL}, Parusi\'nski shows that a regular value $t_0$ for $f$ is not in $B_f$ if and only if the Euler characteristic of 
$f^{-1}(t_0)$ is locally constant at $t_0$ \cite{P}. 
If $f$ is the polynomial of minimal degree defining $D$, we will say 
that $D$ is {\it non-bifurcated} if $0$ is not a bifurcation value for $f$ 
(i.e. $0 \not \in B_f$, meaning $f$ is a trivial fibration in a neighborhood 
of $D=f^{-1}(0)$). More information on the behavior of 
polynomials can be found in \cite{dimca} and \cite{tibar}.

\begin{thm}\label{generic}
Fix $n > 1$ and let $D \subset \Bbb C^n$ be a smooth connected 
non-bifurcated hypersurface. 
If $F:X \to \Bbb C^n$ is a local diffeomorphism of simply connected manifolds
which is a $d$-fold covering map away from $D$, then $d=1$ or $d=\infty$.
\end{thm}
 
For $D \subset \Bbb C^n$ as in \ref{generic}, we highlight two algebro-geometric cases: 

\begin{cor}\label{kulikov's} Let $D \subset \Bbb C^n$ be a smooth connected non-bifurcated hypersurface. 
If $F:X \to \Bbb C^n$ is an \'etale morphism with $X$ simply 
connected and $\#F^{-1}(q)=\deg F$ for $q \not \in D$, then $F$ is birational.
\end{cor}

\begin{cor}\label{jc} Let $D \subset \Bbb C^n$ be a smooth connected non-bifurcated hypersurface. 
If $F:\Bbb C^n \to \Bbb C^n$ is a polynomial map with nonvanishing 
jacobian and $\#F^{-1}(q)=\deg F$ for $q \not \in D$, then 
$F \in {\rm {Aut}}(\Bbb C^n)$. 
\end{cor}
 
The conclusion of \ref{kulikov's} also holds if $D$ has at worst simple normal 
crossings away from a set of codimension $\geq 3$ and meets the 
hyperplane at infinity transversely \cite{NX3}. 
It would be interesting to try to weaken the hypothesis of Theorem \ref{generic}, for example 
what happens if $D$ merely smooth and connected? Or if $D$ is a non-bifurcation hypersurface but reducible? 

To state our main tool for obtaining these results, we need a definition. 
An $(n-1)$-submanifold 
$A \subset \Bbb R^n$ {\bf nicely bounds} a closed subset 
$D \subset \Bbb R^n$ if
\begin{enumerate}

\item $D=\partial A = {\overline A}-A$.
\item Each connected component of ${\overline A}$ contains exactly 
one connected component of $D$.
\item $D$ is the closure of an $(n-2)$-submanifold $D_0 \subset \Bbb R^n$  
with singular locus $\Sigma = D - D_0$ of codimension $\geqslant 4$ in 
$\Bbb R^n$ and ${\overline A}$ is locally diffeomorphic to a half space 
along $D_0$. 

\end{enumerate}

\begin{rmk}\label{existence}{\em
(a) For typical hypersurfaces $D \subset \Bbb C^n$, the existence of a nicely 
bounding manifold $A$ is natural. Suppose that $D$ is defined by the equation 
$P(z_1,z_2,\dots,z_n)=0$ and $S \subset \Bbb C$ is a finite set for which 
the restriction $P: \Bbb C^n - P^{-1}(S) \to \Bbb C - S$ is a locally trivial fibration with fibre $f$ as in Verdier's theorem \cite{verdier}. 
If $f$ is smooth and connected, then in choosing an open ray $l \subset \Bbb C$ 
emanating from $0$ which misses $S$, the contractibility of $l$ implies that 
$A = P^{-1}(l) \cong l \times f$ is an oriented manifold and it is clear that $A$ 
nicely bounds $D$. 
If $D$ is a non-bifurcation hypersurface, then $D \cong f$ and 
${\overline A} \cong {\overline l} \times D$, hence we obtain an 
isomorphism $\pi_1(D) \to \pi_1({\overline A})$.

(b) If $F:X \to \Bbb C^n$ is a $d$-fold cover away from $D$ 
(with equation $P=0$), we can arrange the situation in (a) by enlarging $D$ 
as follows. Replace $P$ with $PQ$, where $Q$ is general and 
$\deg (PQ)$ is prime. Then the hypothesis on $F$ is preserved, 
$D$ is connected and $P$ cannot be be written $h(g(z_1, \dots, z_m))$ 
for  $h(t) \in \Bbb C [t]$ of degree $>1$, which implies that the general 
fibre $f=P^{-1}(t)$ is smooth and connected by a Bertini theorem 
\cite[Cor. 1 to Thm. 37]{schinzel}. Of course this divisor $D$ may be bifurcated. 

\em}\end{rmk}

The degree of such maps often satisfies a certain trichotomy (Thm. \ref{unorient}):  

\begin{thm}\label{counting}
Let $F:X \to \Bbb R^n$ be a local diffeomorphism of connected 
manifolds with $H_1 (X, \Bbb Z)=0$ and let $D \subset \Bbb R^n$ be a 
closed set for which the restriction $X - F^{-1}(D) \to \Bbb R^n - D$ is a 
$d$-sheeted covering map. 
If $D$ is nicely bounded by a connected $(n-1)-$submanifold 
$A \subset \Bbb R^n$ with $\Bbb R^n - {\overline A}$ simply connected, 
then $d=1,2$ or $\infty$. 
\end{thm}

The conclusion can be false if $F$ fails to be a local diffeomorphism at 
even a {\it single point} (visualize the $d$th power map on $\Bbb C$ 
for $d > 2$).

\begin{rmk}\label{oriented}
If the nicely bounding manifold $A$ is oriented, but possibly disconnected, 
then $d=1$ or $d=\infty$ (Theorem \ref{count}). 
\end{rmk}

\begin{rmk}\label{sharp}{\em
Theorem \ref{counting} is sharp as follows. 

(a) All three values of $d$ occur. The identity map realizes $d=1$ 
(for example take $D=\Bbb R^{n-2} \times (0,0)$ and 
$A=\Bbb R^{n-2} \times (0, \infty) \times 0$).
The value $d=2$ is realized by a map $F:X \to \Bbb R^4$ with 
$D \subset \Bbb R^4$ an embedded Klein bottle which is nicely bounded by 
$M \times [-2,2]$, where $M \subset \Bbb R^3$ is a Moebius band (see Example \ref{d=2}). 
The value $d=\infty$ is achieved by the complex exponential map 
$F:\Bbb C \to \Bbb C$ with $D = \{0\}$.

(b) The vanishing $H_1(X, \Bbb Z)=0$ is necessary, for example 
the $d$th power map $F:\Bbb C - \{0\} \to \Bbb C$ is a $d$-sheet 
covering map.

\em}\end{rmk}

\begin{cor}\label{2}
Let $F:X \to \Bbb R^2$ be a local diffeomorphism.
Assume that $X$ is connected, $H_1(X, \Bbb Z)=0$ and the restriction
$F:X - f^{-1}(D) \to \Bbb R^2 - D$
is a $d$-sheeted cover for a finite subset $D \subset \Bbb R^2$.
Then $d=1$ or $d=\infty$.
\end{cor}

Corollary \ref{2} follows from Remark \ref{oriented} by taking $A$ to be the 
union of disjoint open rays emanating from points in $D$, when
$\Bbb R^2 - {\overline A}$ is contractible.
For hypersurfaces $D \subset \Bbb C^n$ with $n>1$ it is unlikely that the 
submanifold $A$ produced in Remark \ref{existence} will satisfy 
$\pi_1(\Bbb C^n - {\overline A})=0$, hence Theorem \ref{counting} does not apply.  
In typical situations surgery theory \cite{ranicki} can be used to kill homotopy 
classes of maps and we can modify $A$ to arrive at a suitable manifold $B$. 
For this we need more conditions on the nicely bounding manifold $A$. 
\begin{enumerate}
\item $D$ has finitely many connected components $D_i$ contained in 
corresponding components $A_i \subset A$. 
\item If $\Sigma \subset D$ is the singular locus, then 
$\pi_1(D_i - \Sigma) \to \pi_1({\overline {A_i}} - \Sigma)$ is onto. 
\item Each component $A_i$ is orientable with $\pi_1(A_i)$ finitely generated.
\end{enumerate}

\begin{thm}\label{larry}
Suppose that $D$ is nicely bounded by a submanifold $A \subset \Bbb R^n$ 
satisfying the conditions 1-3 above and $n \geq 6$. 
Then there is another such $B$ nicely bounding $D$
such that $\pi_1(\Bbb R^n - {\overline B})=0$.
\end{thm}

Now Theorem \ref{generic} follows easily. 
For $n \geq 3$ and $F: X \to \Bbb C^n$, we produce an orientable 
submanifold $A \subset \Bbb C^n$ nicely bounding $D$ as in 
Remark \ref{existence}. 
The hypersurface $D$ has finitely generated fundamental group 
because it is a finite CW complex \cite{dimca}. 
Since there is only one connected component and $\Sigma$ is empty, 
Remark \ref{existence}(a) gives conditions 1-3 above. 
Theorem \ref{larry} allows us to replace nicely bounding $A$ 
with $B$ satisfying $\pi_1(\Bbb C^n - {\overline B})=0$, at which point 
Remark \ref{oriented} gives the result. 
If $n=2$, we consider 
$F \times {\rm {Id}}: X \times \Bbb C \to \Bbb C^n \times \Bbb C$ and 
replace $D$ with $D \times \Bbb C$. 

We expect the surgery technique to work even if the hypersurface $D$ is bifurcated, provided that the singularities are sufficiently mild.
These problems fit in with our general program of trying to understand
local versus global injectivity \cite{NX1,NX2,NX3,X}.
Indeed, for such maps with finite fibres, we are really asking when
a local immersion is injective. 
As to the structure of the paper, Theorem \ref{counting} is proven in
$\S$ \ref{thinsets} and Theorem \ref{larry} is proven in $\S$ \ref{Larry}.
We thank Francis Connolly for useful conversations. 

\section{A Counting Argument}\label{thinsets}

In this section we prove Theorem \ref{counting} from the introduction.
The dichotomy in the cases $d=1$ and $d=\infty$ depends on the nature 
of the pre-image $F^{-1}(A) \subset X$. We will show that if no component of $F^{-1}(A)$ is closed, then $d=1$, while if at least one component is closed, then $d=\infty$. 

\begin{prop}\label{d=1}
Let $f: X \to \Bbb R^n$ be a local diffeomorphism of connected manifolds which is a $d$-sheet cover away from a closed set $D$. Further assume that there is a $(n-1)$-submanifold $A \subset \Bbb R^n$ such that
\begin{enumerate}
\item $D$ is nicely bounded by $A$.
\item $\Bbb R^n - {\overline A}$ is simply connected.
\item No connected component of $f^{-1}(A)$ is closed in $X$.
\end{enumerate}
Then $d=1$.
\end{prop}

\begin{proof}
It will suffice to show that $X-f^{-1}({\overline A})$ is path-connected,
for then $X-f^{-1}({\overline A}) \to \Bbb R^n - {\overline A}$ is a
$d$-sheeted covering map with $\Bbb R^n - {\overline A}$ simply connected, whence $d=1$ \cite[Ch. 2]{spanier}.
In particular, $f$ is injective and identifies $X$ with a dense open subset of $\Bbb R^n$.

To see that $X-f^{-1}({\overline A})$ is path-connected, consider
$a \neq b \in X-f^{-1}({\overline A})$ and consider
a path $\tau: [0,1] \to X$ from $a$ to $b$.
Deforming $\tau$, we may assume that
\begin{enumerate}
\item[(a)] $\tau$ avoids $f^{-1}(D)$.
\item[(b)] $\tau$ meets $f^{-1}(A)$ transversely a finite number of times.
\end{enumerate}
Suppose that $\tau$ meets a connected component $E$ of $f^{-1}(A)$, say
$\tau(t_{0}) \in E$.
Since $E$ is not closed by hypothesis, there is a point
$q \in {\overline E} - E \subset f^{-1}(D)$ which necessarily satisfies $f(q) \in \partial A = D$.
Taking an open neighborhood $U$ about $q$ which maps diffeomorphically onto the neighborhood $f(U)$ about 
$f(q) \in D$, we have a diffeomorphism of pairs 
$({\overline E} \cap U,U) \cong 
({\overline A} \cap f(U),f(U))$ and since 
$\Sigma \subset D$ has positive codimension, we can 
choose $q \in U$ so that $f(q)$ avoids $\Sigma$ and 
${\overline E}$ looks locally like a half-space at $q$.
With this in mind, consider a path $\sigma$ from $\tau(t_{0})$ to
the point $q \in {\overline E}-E$.
Cover the image of $\sigma$ with small open balls to obtain a tubular open neighborhood $N$ of $\sigma$.
A ball about $q$ is not disconnected by ${\overline E}$ and it
follows that $N-{\overline E}$ itself is connected.
We may thus replace a segment of the path $\tau$ through
$\tau(t_{0})$ with a path contained in $N-{\overline E}$.
Continuing this way with each intersection, we obtain a new path from
$a$ to $b$ which avoids $f^{-1}({\overline A})$, hence the conclusion.
\end{proof}

\begin{rmk}{\em
That $f$ be a local diffeomorphism at each point is critical in the result above. 
For example, the squaring map $f: \Bbb C \to \Bbb C$ by $f(z)=z^2$ is a local 
diffeomorphism everywhere except at $z=0$. 
Taking $D = \{0\}$ and $A$ to be the positive real axis, the other hypotheses of 
Proposition \ref{d=1} hold, but $d=2$ and the conclusion fails. 
The proof given above breaks down because $f^{-1}(A)$ is obtained from removing 
the origin from the real axis, when it is clear that $\Bbb C - f^{-1}({\overline A})$ is 
not connected: we can't connect the upper half plane and lower half plane by 
going around the corner at $z=0$ precisely because $f$ is not a local diffeomorphism 
there.
\em}\end{rmk}

Now we deal with the case where some of the connected components of
$f^{-1}(A)$ are closed. We begin with the following separation lemma.
The example $X = S^1$ shows the necessity of the vanishing $H_1(X)=0$.

\begin{lem}\label{separate}
Let $X$ be a connected $n$-manifold satisfying $H_{1}(X, \Bbb Z)=0$.
If $E_{1},\dots,E_{k}$ are $k$ disjoint closed connected
$(n-1)$-submanifolds
of $X$, then
$X-\bigcup_{i=1}^{k} E_{i}$ has $k+1$ connected components $U_{1},\dots,U_{k+1}$.
\end{lem}

\begin{proof}
This is result is probably well-known to experts (the case $k=1$ can be 
found in \cite{dold}), but we include a proof for lack of suitable reference.
The vanishing $H_{1}(X, \Bbb Z)=0$ implies that $X$ is orientable
\cite[VIII, 2.12]{dold}.
Let $Y = \bigcup E_{i}$.
If ${\overline H}_{c}^{q}(A,B)$
denotes the Alexander cohomology module with compact supports
and coefficients in $\Bbb Z$ \cite[\S 6.4]{spanier},
the corresponding long exact cohomology sequence contains the fragment
\[
{\overline H}_{c}^{n-1}(X) \to {\overline H}_{c}^{n-1}(Y) \to
{\overline H}_{c}^{n}(X,Y) \to {\overline H}_{c}^{n}(X) \to
{\overline H}_{c}^{n}(Y).
\]
The last group vanishes due to dimension and
the Alexander duality theorem \cite[6.9.11]{spanier}
yields
${\overline H}_{c}^{n-1}(X) \cong H_{1}(X)=0$,
${\overline H}_{c}^{n-1}(Y) \cong H_{0}(Y) \cong \Bbb Z^{k}$,
${\overline H}_{c}^{n}(X) \cong H_{0}(X) \cong \Bbb Z$.
It follows that
$H_{0}(X-Y)={\overline H}_{c}^{n}(X,Y) \cong \Bbb Z^{k+1}$,
hence $U=X-Y$ has $k+1$ connected components.
\end{proof}

\begin{prop}\label{infinity}
Let $f: X \to \Bbb R^n$ be a local diffeomorphism of connected manifolds
which is a $d$-sheet cover away from a closed set $D$ and assume that
$H_1(X, \Bbb Z)=0$.
Assume that there is an oriented $(n-1)$-submanifold
$A \subset \Bbb R^n$ such that
\begin{enumerate}
\item $D$ is nicely bounded by $A$.
\item $f^{-1}(A)$ has a connected component which is closed in $X$.
\end{enumerate}
Then $d=\infty$.
\end{prop}

\begin{proof}
We assume that $d < \infty$ and draw a contradiction from the counting argument 
that gives this section its name.
Let $E$ be a connected component of $f^{-1}(A)$ which is closed in $X$.
Then $f(E)$ is contained in some connected component $A_1$ of $A$ and the 
covering property shows that $f(E)=A_1$. 
The closed sub-manifold $E \subset X$ separates $X$ into two connected components 
$U_1$ and $U_2$ by Lemma \ref{separate}. Clearly the induced maps
\[
U_i - f^{-1}({\overline A}) \to \Bbb R^n - {\overline A}
\]
have the covering property with sheet numbers adding to $d$, 
but $U_i - f^{-1}({\overline A})$ may be disconnected if 
$f^{-1}(A)$ has closed components other than $E$.

Choose $q \in A_1$ and write $f^{-1}(q) = \{p_1, \dots p_d\}$ with
$p_1, \dots, p_r \in E$ for some $1 \leq r \leq d$.
If $B$ is a sufficiently small ball centered at $q$, then by the covering 
property we may write
\[
f^{-1}(B) = B_1 \cup B_2 \cup \dots \cup B_d
\]
disjointly with each $B_i$ mapping homeomorphically onto $B$ and
$p_i \in B_i$.
Thus $B_1, \dots B_r$ meet $E$ and each $B_{r+1}, \dots B_d$ avoids $E$, 
hence is contained in either $U_1$ or $U_2$.
We may relabel the indices so that $B_{r+1} \dots B_s$ are contained in
$U_1$ and $B_{s+1} \dots B_d$ are contained in $U_2$.

Shrinking $B$ if necessary, we may further assume that
$B \cap A_1$ splits $B$ into two half balls $B^+$ and $B^-$ and that $B$ 
does not meet $A$ otherwise (in particular it intersects no other components
of $A$).
For $1 \leqslant i \leqslant r$, each $B_i$ is separated by $E$ into two
connected pieces $B^+_i$ and $B^-_i$, with $B^+_i$ (resp. $B^-_i$)
mapping homeomorphically onto $B^+$ (resp. $B^-$).

Since $E$ separates $X$, either $B^+_1 \subset U_1$ or $U_2$. We suppose that 
$B^+_1 \subset U_1$ and $B^-_1 \subset U_2$, but the proof is similar if it were the other 
way around.
From the orientability of $A_1$, it follows that the pre-images of the half balls
are coherently arranged in the sense that $B^+_i \subset U_1$ (resp. $B^-_i \subset U_2$) 
for all $1 \leqslant i \leqslant r$.
To see this, consider a path $\gamma(t)$ from $p_1$ to $p_i$ for
$0 \leqslant t \leqslant 1$.
If $N(t)$ is a vector field along $\gamma$ which points into $U_1$, 
then the vector field $df(N)$ along the loop $f \circ \gamma$ points into 
$B^+$ at time $t=0$. By invariance of local orientations under continuations 
along a closed path, $df(N)$ also points into $B^+$ at time $t=1$, which implies 
that $B_i^+ \subset U_1$. Similarly the $B^-_i$ are contained in $U_2$.

Now we arrive at a contradiction by counting pre-images.
A point in $B^+$ has $s$ pre-images in $U_1$:
$r$ of these come from the $B_i^+$ with $1 \leqslant i \leqslant r$ and the remaining 
$s-r$ come from $B_i \subset U_1$ for $r < i \leqslant s$.
Similarly a point in $B^-$ has $s-r$ pre-images in $U_1$, all coming from 
$B_i^-$ with $r < i \leq s$.
In particular the cardinality of the fibres of 
$U_1 - f^{-1}({\overline A}) \to \Bbb R^n - {\overline A}$ is not constant 
($s \neq s-r$), a contradiction as this is a covering over a connected base.
\end{proof}

\begin{thm}\label{count}
Let $X$ be a connected manifold satisfying $H_1 (X, \Bbb Z) =0$ and let
$F: X \to \Bbb R^n$ be a local diffeomorphism which is a $d$-fold
cover away from a closed codimension two subset $D \subset \Bbb R^n$.
Suppose that there is an oriented $(n-1)$-submanifold $A \subset \Bbb R^n$ such that
\begin{enumerate}
\item $D$ is nicely bounded by $A$.
\item $\Bbb R^n - {\overline A}$ is simply connected.
\end{enumerate}
Then $d=1$ or $d=\infty$.
\end{thm}

\begin{proof} This follows from Propositions \ref{d=1} and \ref{infinity}, since either 
$f^{-1}(A)$ has a closed connected component or it does not. 
\end{proof}

\begin{rmk}{\em
If $D$ is given by the vanishing of $P(z_1, \dots z_n)$ and the hypersurface 
$P(z_1, \dots z_n)=c$ is simply connected for general $c \in \Bbb C$, then we can 
construct a suitable submanifold $A$ as follows. 
There is a finite point set $B = \{b_1, \dots, b_s\} \subset \Bbb C$
such that $\Bbb C^n - P^{-1}(B) \to \Bbb C - B$ is a locally trivial fibration with 
fibre $F$. Enlarge $B$ to assume that $0 \in B$.
Now choose disjoint open rays $l_i$ emanating from the $b_i$.
Set $L = \bigcup l_i$ and $A = P^{-1}(L)$.
Then $D = P^{-1}(B)$ has $s$ connected components and is nicely bounded by $A$.
Furthermore, since $\Bbb C^n - P^{-1}({\overline A}) \to \Bbb C - {\overline A}$ is a 
locally trivial fibration over a contractible base with simply connected fibre $F$, 
it follows that $\Bbb C^n - P^{-1}({\overline A})$ is simply connected. 
Thus Theorem \ref{count} applies and $d=1$ or $d=\infty$.
\em}\end{rmk}

Our methods also work when $A$ is not orientable, but connected. 

\begin{thm}\label{unorient}
Let $X$ be a connected manifold satisfying $H_1(X)=0$ and let 
$F:X \to \Bbb R^n$ be a local diffeomorphism which is a $d$-fold cover away 
from a closed subset $D \subset \Bbb R^n$. 
Suppose that there is a connected $(n-1)$-submanifold $A \subset \Bbb R^n$ such that
\begin{enumerate}
\item $D$ is nicely bounded by $A$.
\item $\Bbb R^n - {\overline A}$ is simply connected.
\end{enumerate}
Then $d=1,2$ or $\infty$.
\end{thm}

\begin{proof}
If no connected component of $f^{-1}(A)$ is closed, then Proposition \ref{d=1} still 
applies and we obtain $d=1$. If there is a closed such component, we must 
modify the proof of Proposition \ref{infinity}. 
Here we consider {\it all} the closed connected components 
$E_1, E_2, \dots E_m$ of $f^{-1}(A)$. 
By Lemma \ref{separate} these separate $X$ into $m+1$ connected components 
$V_1, V_2, \dots V_{m+1}$. 
Now each $V_i - f^{-1}({\overline A}) \to \Bbb R^n - {\overline A}$ has the covering 
property.
Moreover, we may argue as in Prop. \ref{d=1} that each $V_i - f^{-1}({\overline A})$ 
is connected: that $f$ is a local diffeomorphism allows us to ``go around the edge" in the 
same way. 
Since the base is $1$-connected, we conclude that these maps are homeomorphisms and $m+1=d$. 

Letting $d_i$ be the covering degree of the restriction $f: E_i \to A$, we have 
$\sum_{i=1}^{d-1} d_i \leq d$.
If $d > 2$, this forces $d_i=1$ for some $i$ (since $d_i \geq 1$ already), say $i=1$.
Then $E_1 \cong A$ via $f$ and we run the proof of Proposition \ref{infinity} with 
$E = E_1$.
Things proceed as before, though now $r=1$. 
We don't need to use the orientability of $A$ to check that the $B_i^+$ can be 
coherently arranged for $1 \leq i \leq r$, since there is only one. 
Thus the counting contradiction runs as before, as the preimage of a point 
in $B^+$ has $s$ pre-images in $U_1$ while a point in $B^-$ has $s-r=s-1$ such pre-images.
Hence $2 < d < \infty$ gives a contradiction, leaving $d=1,2$ or $\infty$. 
\end{proof}

\begin{rmk}\label{d=2}{\em
Comparing to Theorem \ref{count}, we see that when $A$ is not 
orientable the new possibility $d=2$ arises. 
We give an example in which this new possibility is realized. 

Fix an unknotted circle $C_0 \subset \R^3$ and consider its normal 
disk bundle, which is a solid torus. Inside the solid torus there is a Moebius 
band $M$ which makes one half twist as it goes around.
The boundary $C_\partial = \partial M$ of this Moebius band is a 2-1 torus 
knot which is known to be the unknot.
We will show that $M \times (-2,2) \subset \Bbb R^4$ nicely bounds a closed set $K$, 
which can be smoothed to obtain a Klein bottle $K_1$ for which $\pi_1(\Bbb R^4 - K) \cong \cy2$. 
In particular, the composition $X \to \Bbb R^4 - K \hookrightarrow \Bbb R^4$ of the universal 
covering map with the inclusion has degree $d=2$. 

To verify the claims above, let $H_+=\R^3 \times[0,\infty) \subset \Bbb R^4$ be the 
4-dimensional half space.
For each integer $n>0$ we set $t_n^\pm = 2\pm\frac{1}{n}$ and define the 
manifold $W_n \subset H_+$ by 
\[
W_n = \bigl((\R^3 - C_\partial)\times\left[0,t_n^-\right] \bigr) \cup
\bigl(\R^3-M)\times\left[t_n^-,t_n^+\right]\bigr) \cup
\bigl(\R^3\times\left[t_n^+,\infty\right)\bigr).
\]
Notice that 
$W = \displaystyle\mathop{\cup}_{n=1}^\infty W_n = 
H_+-(C_\partial \times [0,2] \cup M \times \{2\})$ 
is just $H-+$ minus a Moebius band.
It is a bit angular, but we will round off the corners later. 
More importantly, the inclusion $W_n\subset W_{n+1}$ is a deformation retract and 
hence a homotopy equivalence: thus $\pi_1(W)$, as the limit of the $\pi_1(W_n)$, 
is isomorphic to $\pi_1(W_n)$ for any $n$.

We compute $\pi_1(W_1)$ as follows.
Apply Van Kampen's theorem to  
\[
V=\bigl((\R^3 - C_\partial)\times [0,1] \bigr) \cup \bigl( (\R^3 - 
M)\times [1,3]\bigr).
\]
The intersection of the two pieces is $(\R^3 - M) \times 1$ and the two inclusions
$(\R^3 - M) \times 1 \subset (\R^3 - M) \times [1,3]$ and 
$(\R^3 - C_\partial)\times 0 \subset (\R^3 - C_\partial) \times [0,1]$
are homotopy equivalences. 

Since $C_\partial$ is unknotted, $\pi_1(\R^3 - C_\partial) \cong \Z$ and
a generator is any small circle $L_\partial$ linking $C_\partial$.
Since $\pi_1(\R^3 - C_\partial)$ is Abelian, base points are not important.
Recall $C_0\subset M$ is the middle circle in the Moebius and note 
that the inclusion $\R^3 - M \subset \R^3 - C_0$ is a homotopy equivalence.
Since $C_0$ is unknotted, $\pi_1(\R^3 - M) \cong \Z$. One choice 
of generator is to pick a meridian circle to the solid torus containing $M$ and push 
it out just a bit until it no longer intersects the solid torus, call this circle $L_0$.
Since $\pi_1(\R^3 - M)$ is abelian, base points are not important here.

Finally compute the map $\pi_1(\R^3 - M) \to \pi_1(\R^3 - C_\partial)$.
The obvious disk bounding the generator of  $\pi_1(\R^3 - M)$ intersects
$C_\partial$ in two points, both with the same sign, hence the map is multiplication 
by $\pm 2$: $L_0 = \pm 2 L_\partial$.
Thus $\pi_1(V)\cong \Z$ with generator $L_\partial$.

To finish, we apply Van Kampen's theorem to 
\[
W_1 = V \cup \bigl(\R^3\times[3,\infty)\bigr).
\]
The intersection is $\bigl((\R^3 - M)\times 3\bigr)$.
Since $\pi_1\bigl(\R^3\times[3,\infty)\bigr) = 0$, it follows that
$\pi_1(W_1) = \pi_1(V) / \pi_1\bigl((\R^3 - M)\times 3\bigr)$ which 
by the above calculation is $\cy{2}$. 
We conclude that $\pi_1(W) = \pi_1(W_1) = \cy{2}$
generated by a circle in $(\R^3 - C_\partial)\times 0$ linking $C_\partial$.

Having produced $W = W_+ \subset H_+$ above, we mirror the constuction to obtain
$W_- \subset H_- = \R^3 \times(-\infty,0]$ whose complement is a Moebius band. 
Gluing $H_+$ to $H_-$ along $\R^3 \times 0$ gives $\R^4$ which contains 
$M \times [-2,2]$: let $\iota\colon M\times [-2,2] \hookrightarrow \R^4$ be the embedding.
The boundary of $M\times [-2,2]$ is a non-orientable surface of Euler
characteristic $0$, hence a Klein bottle.
Let $K = \iota\bigl(\partial(M\times[-2,2])\bigr)$.
By construction $H_\pm - K = W_\pm$ and again Van Kampen's theorem gives 
$\pi_1(\R^4 - K) = \cy{2}$.

This $K$ is certainly not smooth but by rounding corners we can make 
it smooth.
A direct approach in this example is to find a $C^\infty$ function
$f\colon M \to \left[\smfrac{1}{2},1\right]$ which is $\smfrac{1}{2}$
on a neighborhood of $\partial M$ and $1$ on a neighborhood of the core.
Let $F\colon M\times[-2,2] \to M\times[-2,2]$ be defined by $F(m,t) = 
\bigl(m,f(m)\cdot t\bigr)$
Let $W = F\bigl(M\times[-2,2]\bigr)\subset M\times\R$ be the image of 
$F$.
The boundary of $W$ is still a Klein bottle and $W$ is a $C^\infty$ 
submanifold of
$M\times \R$.
Then the composition
$\bar{\iota}\colon W \subset M\times\R \rta{\iota} \R^4$ is $C^\infty$.
Hence $K_1=\bar{\iota}\bigl(\partial W\bigr)$ is a smoothly embedded
Klein bottle, bounding a smoothly embedded $3$-manifold with
$\pi_1(\R^4 - K_1) = \cy{2}$.
\em}\end{rmk}

\section{Construction of bounding manifolds}\label{Larry}

In this section we prove Theorem \ref{larry} from the introduction. 
The main point is Lemma \ref{reduction}, which uses surgery to annihilate elements in the fundamental group of the bounding manifold. We will use the following general set up.
\begin{hyp}\label{fixed hyp}
Let $M$ be an $n$-manifold with $(n-1)$-submanifold $A \subset M$ and closed subset $D$. The triple $(M,D,A)$ satisfies Hypothesis \ref{fixed hyp} if the following hold.
\begin{enumerate}
\item[(a)] $M$ and $A$ are orientable.
\item[(b)] $D$ has finitely many components and is nicely bounded by $A$. In particular $D$ is the closure of an $(n-2)$-submanifold $D_0$ with singular set $\Sigma = D - D_0$ of codimension $\geq 4$ in $M$. 
\gdef\zeroConn{(a)}
\item[(c)] Each component of $A$ has finitely generated fundamental group.
\gdef\pioo{(d)}
\item[(d)] For each $x_0 \in D-\Sigma$, the map
$\pi_1(D-\Sigma,x_0) \to \pi_1(\,{\overline A}-\Sigma ,x_0)$ is onto.
\end{enumerate}
\end{hyp}

\begin{rmks}{\em
Since ${\overline A} - \Sigma$ is a manifold with boundary,
$A \subset {\overline A} - \Sigma$ is a homotopy equivalence.
Since $\Sigma$ has codimension $2$,
$\pi_1(D - \Sigma,x_0)\to \pi_1(D,x_0)$ is onto and
$\pi_1(M - \Sigma,x_0) \to \pi_1(M,x_0)$ is an isomorphism.
If $M$ is connected, so is $M - {\overline A}$.}
\end{rmks}

We will need the following definition.

\begin{defn}{\em
If $A$ and $B$ are submanifolds of $M$, we say that $A$ and $B$ are \emph
{equal at $\infty$}
provided there exists a compact set $C\subset M$ such that
$A-C$ and $B-C$ are equal as subsets of $M - C$.}
\end{defn}

\begin{thm}\label{surgery thm}
If the triple $(\R^n,D,A)$ satisfies \ref{fixed hyp} and $n\geqslant 6$,
then there exists a submanifold $B \subset \R^n$ such that $(\R^n, D, B)$ 
also satisfies \ref{fixed hyp}, $B$ is equal to $A$ at $\infty$ and
$$\pi_1\bigl(\,\R^n - {\overline {B}}\bigr) = 0.$$
\end{thm}
\bigskip

\begin{proof}
We reduce the proof to the two lemmas which follow. 
Remove the singular set $\Sigma$ to obtain the new triple 
$(M,D_0,A)=(\Bbb R^n-\Sigma, D-\Sigma, A)$, 
which clearly satisfies \ref{fixed hyp}.
Lemma \ref{reduction} allows us to replace $A$ with a new submanifold 
$A^\prime$ which is equal to $A$ at infinity and such that 
$\pi_1(A^\prime)$ requires one fewer generators than $\pi_1(A)$ 
(by which we mean the free product of $\pi_1(A_i,x_i)$, where one 
basepoint $x_i$ is chosen from each connected component $A_i$, as $A$ 
may be disconnected).
Because $A^\prime$ and $A$ are equal at infinity, the new triple
$(M,D_0,A^\prime)$ satisfies \ref{fixed hyp} and
$\R^n - {\overline {A^\prime}} = M - {\overline {A^\prime}}$.
Note the closures take place in different sets: the left-hand closure 
is equal to the right-hand closure union $\Sigma$.
Repeated application reduces to the case that $\pi_1(A)$ is trivial, 
at which point Lemma \ref{simply-connected A} shows that
$\pi_1(M-{\overline A})=0$.
\end{proof}

\begin{lem}\label{simply-connected A}
Suppose that $(M, D_0, B)$ satisfies \ref{fixed hyp} with 
$\Sigma = \emptyset,\; \pi_1(M) = 0$ and $\pi_1(B) = 0$.
Then $\pi_1\bigl(M - {\overline B}\bigr)=0$.
\end{lem}

\begin{proof}
Let $e\colon S^1 \to M - {\overline B}$ represent an 
element of $\pi_1(M - {\overline B})$. It suffices to show that 
$e$ extends to a map of the disk $D^2 \to M - {\overline B}$.

By hypothesis we can extend $e$ to a map $e\colon D^2 \to M$ and
we may assume that $e$ is an embedding since $n\geqslant 6$.
The set ${\overline B}$ is a manifold with boundary $D_0$, so
we may further assume that $e\colon D^2 \to M$ is transverse 
to ${\overline B}$.
Now $D^2 \cap e^{-1}{\overline B}$ is a closed, compact $1$-manifold, 
hence a disjoint union of circles and closed arcs in the interior of $D^2$.
The transverse condition also guarantees that $D^2 \cap e^{-1}B$ 
consists of circles
and open arcs and $D^2 \cap e^{-1}(D_0)$ consists entirely of points which 
are the end points of the arcs: no image of a circle is tangent to $D_0$. 

Each circle in $D^2$ divides the disk into two components, an inside
(homeomorphic to a disk) and an outside (homeomorphic to an annulus).
Pick an outer-most circle $C \subset D^2$ from $e^{-1}(B)$, not contained
on the inside of any other circle of $e^{-1}(B)$.
The embedded circle $e(C)$ in $B$ is null homotopic in $B$ by hypothesis. 
Since $\dim B \geq 5$ and $e(D^2) \cap B$ 
consists of a finite number of open arcs and circles, this circle bounds an embedded disk $\Delta \subset B$ such that $\Delta \cap e(D^2) = e(C)$.

The normal bundle $\nu$ of $B$ in $M$ is a real line bundle which 
restricted to $\Delta$ is trivial.
The normal bundle to $C$ in $D^2$ is also trivial and picking the 
direction into the outside
annulus gives a trivialization of $\nu$ restricted to $e(C) = 
\partial \Delta$.
This trivialization extends across $\Delta$ and it is now possible to 
cut out a neighborhood
of the inside of $C$ and glue in a copy of $\Delta$ pushed off of $B$ (in the direction of $\nu$) so that the new embedded disk has fewer circles.
Since there are only finitely many circles, continuing this process
will result in a new embedding $e:D^2 \to M$ for which 
$D^2 \cap e^{-1}({\overline B})$ consists entirely of arcs in 
the interior of $D^2$.

Let $\alpha \in {\overline B}$ be the image of an arc, so the 
interior of the arc is in $B$ and the two endpoints are in $D_0$.
Both endpoints are in the same component of $B$ and so by 
\ref{fixed hyp}(b), they are in
the same component of $D_0 = \partial {\overline B}$, hence 
they may be joined by an arc in $\partial {\overline B}$. 
The resulting circle in ${\overline B}$ is null homotopic because $B \subset {\overline B}$ is a homotopy equivalence. 
Since $n\geqslant 6$ there is an embedded disk 
$\Delta \subset {\overline B}$ so that 
$\partial \Delta \cap \partial {\overline B}$ is an arc in 
$\partial {\overline B}$ and
$\partial \Delta \cap B$ is the open interior of the arc $\alpha$.
We can further assume that $\Delta \cap e(D^2) = \alpha$, 
so that no part of $e(D^2)$ goes near $\Delta$ except for the arc under discussion.

Enlarge $\Delta$ to a disk $\Delta_1$ with $\Delta$ in its interior so that $\Delta_1 \cap \partial{\overline B}$ is a single arc, which must extend the 
arc $\partial D \cap \partial{\overline B}$. 
This arc divides $\Delta_1$ into two pieces, one in $B$ and one in 
$M - {\overline B}$. 

The normal bundle to $\alpha$ in $e(D^2)$ is trivial, so pick a trivialization. 
By transversality this picks out a nonvanishing section of $\nu$ restricted to $\alpha$, which extends to a section of $\nu$ restricted to 
$\Delta_1 \cap {\overline B}$.
Viewing this as a section of the normal bundle to 
$\Delta_1 \cap {\overline B}$ in $M$, we may extend it to obtain 
a section $\tau$ of the normal bundle $\mu$ of $\Delta_1$ in $M$ 
which is normal to ${\overline B}$ along the common intersection.  

Now choose an isotopy of $\Delta_1$ under which $\Delta$ ends up in 
the piece of $\Delta_1$ in $M - {\overline B}$. Extend this to an isotopy 
of $M$, using the section $\tau$ of $\mu$ along $\Delta_1$ to make the movement locally parallel to ${\overline B}$. 
When we apply this isotopy to $e(D^2)$, we remove the arc $\alpha$ 
and create no further intersections, thus the new embedding of the disk 
has the same boundary and intersects ${\overline B}$ in no circles and 
one fewer arc. Continue until there are no arcs (or circles). 
\end{proof}

\begin{lem}\label{reduction}
Suppose that $(M, D_0, A)$ satisfies \ref{fixed hyp} with $\Sigma=\emptyset$ and $\pi_1(A)$ is generated by $m$ elements.
Then there is a submanifold $B \subset M$ such that $(M, D_0, B)$ 
satisfies
\ref{fixed hyp}, $A$ and $B$ are equal at $\infty$ and
$\pi_1(B)$ is generated by at most $m-1$ elements.
\end{lem}

\begin{proof}
Let $g_1, g_2, \dots g_m \in \pi_1(A)$ be generators.
Since $\dim A \geq 5$, the element 
$g_1$ can be represented by an embedded circle $\gamma$ on some component of $A$.

The normal bundle $\nu$ to ${\overline A}$ in $M$ is a trivial line bundle. 
Use this trivialization to embed an annulus in $M$ intersecting 
${\overline A}$ precisely in $\gamma$. 
The embedded circle $\gamma^\prime$ at the other end of the annulus is null homotopic in $M$ and extends to an embedded disk 
$\Delta \subset M$ because $M$ has dimension $ \geq 6$. 
Clearly $\Delta \cap {\overline A}$ contains $\gamma$ by construction, but it may contain more. 
We may assume that $\Delta$ meets ${\overline A}$ transversely away from $\gamma$, hence $\Delta \cap {\overline A}$ consists of $\gamma$ along 
with a finite number of circles and arcs in the interior of $\Delta$.

This time we deal with the arcs first.
Let $\alpha$ be an arc in $\Delta \cap {\overline A}$.
It lies in a component of ${\overline A}$, so both endpoints 
are in a component of $D_0 = \partial {\overline A}$ and can 
be connected by an arc $\tau$ in $D_0$. 
Choosing one endpoint $x_0$ as basepoint, condition \ref{fixed hyp}(d)  
gives a loop $\rho$ based at $x_0$ in $D_0$ such that $\rho = \alpha \cup \tau$ in $\pi_1({\overline A},x_0)$. Replacing $\tau$ with 
$\alpha^\prime = \tau \rho^{-1}$ gives an arc in $\partial{\overline A}$ for which the embedded circle $\theta = \alpha \cup \alpha^\prime$ is null homotopic in ${\overline A}$.
We may further assume that $\alpha^\prime$ intersects the finite set 
$\partial{\overline A} \cap \Delta$ only in the two endpoints of 
$\alpha$. 

Consider the open annulus $X = \Delta - \gamma - \alpha$: 
$\Delta - \gamma$ is an open $2$-disk and removing the 
arc $\gamma$ is the topologically the same as removing a point.
Then $X\cap{\overline A}$ is a finite union of circles and closed 
arcs so $\theta$ is null homotopic in 
${\overline A} - \bigl(X\cap {\overline A}\bigr)$ 
(we have removed all the circles and arcs {\it except} $\gamma$ and $\alpha$). 
Since the dimension of ${\overline A}$ is $\geq 5$, $\theta$ bounds an 
embedded disk $\Gamma \subset {\overline A}$ such that 
$\Gamma \cap \Delta = \alpha$.

Just as in the last part of the proof of Lemma \ref{simply-connected A},
we extend $\Gamma$ to a larger embedded disk $\Gamma^\prime$
such that  $\Gamma^\prime \cap \Delta$ is an arc containing $\alpha$.
There is a neighborhood $N$ of $\Gamma^\prime$
such that $N\cap \Delta$ is the original annulus.
The usual Whitney trick isotopy \cite[$\S$ 7.3]{ranicki} produces 
a new embedded disk $\Delta^\prime$ so that 
$\Delta^\prime \cap {\overline A} = 
\bigl(\Delta \cap {\overline A}\bigr) - \alpha$.
Repeating this process for each arc, we may assume that our disk 
$\Delta$ intersects ${\overline A}$ in a finite number of circles: 
all arcs are gone.

So far we have not altered $A$, but to get rid of the circles we will.
Begin by locating an inner-most circle $\gamma^\prime$, one whose interior 
contains no other circles.
Then $\gamma^\prime$ bounds an embedded disk $\Delta^\prime$ 
such that $\Delta^\prime \cap{\overline A} = \gamma^\prime$.

We perform an operation known as \emph{adding a $2$-handle} 
\cite[$\S 2.4$ and $\S 5.4$]{ranicki}.
Specifically, let $\nu$ be the normal bundle to $\gamma^\prime$ in $A$.
By tranversality this is also the restriction of $\nu^\prime$, the 
normal bundle to $\Delta^\prime$ in $M$.
Since $\nu^\prime$ is a bundle on a disk, it is trivial so fix a trivialization. 
This trivializes $\nu$.
Form the union $X = A \cup E(\nu^\prime)$ where 
$E(\nu^\prime)$ is the closed disk bundle associated to 
$\nu^\prime$. 
If $x_0$ is a basepoint in 
$A \cap E(\nu^\prime)=E(\nu)\cong \gamma^\prime \times {\overline D^{n-2}}$, 
Van Kampen's theorem produces a pushout diagram
\[
\begin{array}{ccc}
\pi_1(E(\nu),x_0) = \langle \gamma^\prime \rangle & \to & 
\pi_1(E(\nu^\prime),x_0)=\{1\} \\
\downarrow & & \downarrow \\
\pi_1(A,x_0) & \to & \pi_1(X,x_0)
\end{array}
\]
from which we see that $\pi_1(X,x_0) \cong \pi_1(A,x_0)/\langle \gamma^\prime \rangle$, where $\langle \gamma^\prime \rangle$ is the normal closure of $\gamma^\prime$. 

Now form $B_1$ from $X$ by removing the open disk bundle. 
The result is not a differentiable manifold as there are corners where $A$ meets $\partial E(\nu^\prime)$, but let us note several properties of $B_1$.
Since $\Delta \cap B_1 = \Delta \cap A - \gamma$ by construction, the new intersection has one less circle. 
Since the construction takes place within a compact set, 
$A$ and $B_1$ are equal at $\infty$ and it is clear that 
$(M,D_0,B_1)$ satisfies \ref{fixed hyp} because $A = B_1$ near $D_0$.

Now $\pi_1(B_1) \to \pi_1(X)$ is an isomorphism (by Van Kampen's theorem, for example) and we write 
$\pi_1(B_1) \cong \pi_1(A)/\langle \gamma^\prime \rangle$ so that $\pi_1(B_1)$ is still generated by $g_1 \dots g_m$.
The commutative diagram 
\[
\begin{array}{ccc}
\pi_1(D_0) & \to & \pi_1({\overline {B_1}}) \\
\downarrow & & \downarrow \\
\pi_1({\overline A}) & \to & \pi_1({\overline X})
\end{array}
\]
shows that the top map is onto. Indeed, the left vertical map is onto by hypothesis, the bottom map is onto because $\pi_1(A) \to \pi_1(X)$ is 
(the inclusions $A \subset {\overline A}$ and $X \subset {\overline X}$ are homotopy equivalences) and the right vertical map is an isomorphism because $\pi_1(B_1) \to \pi_1(X)$ is. 

Finally we round the corners of $B_1$ to obtain a differentiable manifold. 
This process is a homeomorphism (though not a diffeomorphism) so we retain all the properties of $B_1$ above. By the same process we remove an inner-most circle of $B_1\cap \Delta$ to get $B_2$ and 
so on until we find a $B_k$ such that $B_k\cap \Delta = \gamma$, 
$(M, D_0, B_k)$ satisfy \ref{fixed hyp} and $B_k = A$ at infinity.

Now we can trade the $2$-handle $\Delta$ to get $B$ where $A$ and $B$ 
are equal at $\infty$, $(M, D_0, B)$ satisfy \ref{fixed hyp} and 
$\pi_1(A)$ maps onto $\pi_1(B)$ with $\gamma$ definitely in the kernel, 
hence $\pi_1(B)$ is generated by $g_2$, \dots, $g_m$.
\end{proof}

\end{document}